\documentclass{amsart}

\usepackage{color,graphicx,amssymb,latexsym,amsfonts,txfonts,amsmath,amsthm}
\usepackage{pdfsync,enumitem}
\usepackage{amsmath,amscd}
\usepackage[all,cmtip]{xy}

\usepackage{hyperref}
\hypersetup{
    colorlinks=true,       % false: boxed links; true: colored links
    linkcolor=blue,          % color of internal links
    citecolor=blue,        % color of links to bibliography
    filecolor=blue,      % color of file links
    urlcolor=blue           % color of external links
}

\input epsf

\newtheorem{theo}{Theorem}[section]

\newtheorem{coro}[theo]{Corollary}
\newtheorem{lemm}[theo]{Lemma}

\newtheorem{prop}[theo]{Proposition}

\theoremstyle{remark}
\newtheorem{rema}[theo]{\bf Remark}
\newtheorem{example}[theo]{\bf Example}

\begin{document}

\title{On the FOD/FOM parameter of rational maps}
\author{Ruben A. Hidalgo}
\address{Departamento de Matem\'atica y Estad\'{\i}stica, Universidad de La Frontera. Temuco, Chile}
\email{ruben.hidalgo@ufrontera.cl}


\thanks{Partially supported by Project Fondecyt 1230001}

\begin{abstract}
Let $\chi$ be a (right) action of ${\rm PSL}_{2}({\mathbb L})$ on the space ${\mathbb L}(z)$ of rational maps defined over an algebraically closed field ${\mathbb L}$.
If $R \in {\mathbb L}(z)$ and ${\mathcal M}_{R}^{\chi}$ is its $\chi$-field of moduli, then the parameter ${\rm FOD/FOM}_{\chi}(R)$ is the smallest integer $n \geq 1$ such that there is a $\chi$-field of definition of $R$ being a degree $n$ extension of ${\mathcal M}_{R}^{\chi}$.
When ${\mathbb L}$ has characteristic zero and $\chi=\chi_{\infty}$ is the conjugation action, then it is known that ${\rm FOD/FOM}_{\chi_{\infty}}(R) \leq 2$.
In this paper, we study the above parameter for general actions and any characteristic. 
\end{abstract}

\maketitle

%%%%%%%%%%%%%%%%%%
%%%%%%%%%%%%%%%%%%
\section{Introduction}
In this paper, ${\mathbb L}$ will denote an algebraically closed field, and ${\rm G}({\mathbb L})$ will be its group of field automorphisms. 
If $\sigma, \tau \in {\rm G}({\mathbb L})$, then we write by $\sigma \tau$ to denote the composition $\sigma \circ \tau$. 
The group ${\rm G}({\mathbb L})$ induces a 
natural (left) action on the field ${\mathbb L}(z)$ of rational maps defined over ${\mathbb L}$:  if $\sigma \in {\rm G}({\mathbb L})$ and
$R \in {\mathbb L}(z)$, then $R^{\sigma} \in {\mathbb L}(z)$ is the rational map obtained by applying $\sigma$ to the coefficients of $R$ (if $\sigma, \tau \in {\rm G}({\mathbb L})$ and $R \in {\mathbb L}(z)$, then $R^{\sigma \tau}=(R^{\tau})^{\sigma}$). This action preserves the natural grading of ${\mathbb L}(z)$ given by the degree of rational maps.

Let $\chi$ be a (right) action of the group of M\"obius transformations (i.e., the group of rational maps of degree one) ${\rm Aut}({\mathbb P}^{1})={\rm PSL}_{2}({\mathbb L})$ on ${\mathbb L}(z)$. We are not requiring the action to preserve the grading. Examples of (right) actions are: (i) (the conjugation action) $\chi_{\infty}(T)(R)=T^{-1} \circ R \circ T$, (ii) (the pull-back action) $\chi_{k}(T)(R)=R(T(z)) (T'(z))^{k}$, where $k \in {\mathbb Z}$. In these cases, the ones preserving the grading are $\chi_{\infty}$ and $\chi_{0}$.

Two rational maps $R,S \in {\mathbb L}(z)$ are $\chi$-equivalent if there is some $T \in {\rm PSL}_{2}({\mathbb L})$ such that $S=\chi(T)(R)$. A subfield ${\mathbb K}$ of ${\mathbb L}$ is called a $\chi$-field of definition of $R \in {\mathbb L}(z)$ if there exists some $S \in {\mathbb K}(z)$ and some $T \in {\rm PSL}_{2}({\mathbb L})$ such that $S=\chi(T)(R)$. Note that $\chi$-equivalent rational maps have the same $\chi$-fields of definition. 

It might happen that, for some $\chi$-equivalent rational maps $R, S \in {\mathbb L}(z)$, and some $\sigma \in {\rm G}({\mathbb L})$, 
the rational maps $R^{\sigma} $ and $S^{\sigma}$ are not $\chi$-equivalent. To avoid this pathology, we will 
assume that the action $\chi$ satisfies the ${\rm G}({\mathbb L})$-property (see Section \ref{Sec:actions}):
for every $T \in {\rm PSL}_{2}({\mathbb L})$, every $\sigma \in {\rm G}({\mathbb L})$ and every $R \in {\mathbb L}(z)$, it holds that $\chi(T)(R)^{\sigma} = \chi(T^{\sigma})(R^{\sigma})$. The actions $\chi_{\infty}$ and $\chi_{k}$ are examples of actions satisfying the ${\rm G}({\mathbb L})$-property.

The ${\rm G}({\mathbb L})$-property, permits us to define the $\chi$-field of moduli ${\mathcal M}_{R}^{\chi}$ of $R$. It holds that $\chi$-equivalent rational maps have the same $\chi$-field of moduli and that ${\mathcal M}_{R}^{\chi}$ is always contained in every $\chi$-field of definition of $R$ (but, in general, ${\mathcal M}_{R}^{\chi}$ is not a $\chi$-field of definition). We define the parameter ${\rm FOD/FOM}_{\chi}(R)$ as the smallest integer $n \geq 1$ such that $R$ has a $\chi$-field of definition ${\mathbb K} \leq {\mathbb L}$ which is an extension of degree $n$ of ${\mathcal M}_{R}^{\chi}$.

The conjugation action $\chi_{\infty}$ has been well-studied and used to understand arithmetic aspects of the dynamics of rational maps (see, for instance, \cite{Manes, Silv1, Silv2}). Assuming ${\mathbb L}$ to be of characteristic zero, 
in \cite{Silv1}, Silverman proved that every rational map $R$, either of even degree or being $\chi_{\infty}$-equivalent to a polynomial, is definable over its $\chi_{\infty}$-field of moduli (so, for these cases, ${\rm FOD/FOM}_{\chi_{\infty}}(R)=1$).
In the same paper, there were provided explicit examples of rational maps, in every odd degree $d \geq 3$, that cannot be defined over their $\chi_{\infty}$-fields of moduli (so, for these examples, 
${\rm FOD/FOM}_{\chi_{\infty}}(R) \geq 2$).
In \cite{Hidalgo}, it was proved that if a rational map (necessarily of odd degree) cannot be defined over its $\chi_{\infty}$-field of moduli, then it can be defined over a suitable quadratic extension of it. As a consequence, for every $R$ it holds that ${\rm FOD/FOM}_{\chi_{\infty}}(R) \leq 2$.
In the same paper, for ${\mathbb L}={\mathbb C}$, it was observed that ${\mathcal M}_{R}^{\chi_{\infty}}$ is contained in the reals if and only if $R$ commutes with an extended M\"obius transformation (i.e., the composition of a M\"obius transformation with complex conjugation) and, in this case, that $R$ can be defined over the reals if and only if such an extended M\"obius transformation can be chosen to be a reflection (i.e., of order two with fixed points).

We may wonder if the above properties, valid for $\chi_{\infty}$, are still true for a general right action $\chi$:
\begin{enumerate}[leftmargin=15pt]
\item[(i)] When is ${\mathcal M}_{R}^{\chi}$ a $\chi$-field of definition for $R$?
\item[(ii)] Does ${\rm FOD/FOM}_{\chi}(R) \leq 2$?
\item If ${\mathbb L}={\mathbb C}$, when does ${\mathcal M}_{R}^{\chi} \leq {\mathbb R}$?
\end{enumerate}

A necessary and sufficient condition for a subfield ${\mathbb K}$ of ${\mathbb L}$ to be a $\chi$-field of definition for $R$ may be stated in terms of the existence of certain $\chi$-coboundaries associated to the group $G({\mathbb L}/{\mathbb K})$ (see Theorem \ref{TeoWeil0}). Unfortunately, detecting the existence of $\chi$-coboundaries is a difficult problem in general.
 But, to detect the existence of $\chi$-cocycles seems to be a more simple task. In the case that ${\mathbb K}$ is a perfect field, and by restricting to certain (right) actions, which we have called universal ones (see Section \ref{Sec:universal}), we observe (see Theorem \ref{teomain}) that the existence of a $\chi$-coclycle for the rational map $R$ asserts the existence of a $\chi$-coboundary either over ${\mathbb K}$ or over a quadratic extension of it (i.e., $R$ is $\chi$-definible over at most a quadratic extension of ${\mathbb K}$). 
 
If the group ${\rm Aut}_{\chi}(R)$ of $\chi$-automorphisms of $R$ (see Section \ref{Sec:actions}) is finite and the $\chi$-field of moduli ${\mathcal M}_{R}^{\chi}$ is perfect (for instance, when ${\mathbb L}$ either has characteristic zero or it is the algebraic closure of finite field), then ${\rm FOD/FOM}_{\chi}(R) \leq 2$ (Theorem \ref{coromain}).

As the conjugation action $\chi_{\infty}$ is an example of a universal one, the above result extends the results in \cite{Hidalgo} to positive characteristic. 

As the pull-back action $\chi_{k}$, for $k \in {\mathbb Z}$, is also a universal one, 
 if ${\mathcal M}_{R}^{\chi_{k}}$ is a perfect field, then ${\rm FOD/FOM}_{\chi}(R) \leq 2$
 (for $k=0$ and ${\mathbb L}={\mathbb C}$ this fact was obtained by Couveignes \cite{Couv} in connection to Belyi maps and Grothendieck's dessins d'enfants \cite{Gro}). 

As previously observed, every rational map, either of even degree or $\chi_{\infty}$-equivalent to a polynomial, can be defined over its $\chi_{\infty}$-field of moduli. In contrast, 
we provide explicit examples of even degree rational maps $R$ (of degree $4k$ if $k \neq 0$, and degree $4$ if $k=0$), for which ${\mathcal M}_{R}^{\chi_{k}}$ is not a $\chi_{k}$-field of definition (Proposition \ref{ejemplo2}) making a different to the case of $\chi_{\infty}$. 
In these provided examples, the pole divisor of the $k$-form $\omega_{R}=R(z) dz^{k}$ has an even degree. We observe that this is the only obstruction for $R$ to be $\chi_{k}$-definable over ${\mathcal M}_{R}^{\chi_{k}}$ (see Corollary \ref{suficiente}).

Also, for ${\mathbb L}={\mathbb C}$, we observe that ${\mathcal M}_{R}^{\chi_{k}}$ is contained in the field ${\mathbb R}$ of real numbers if and only if the $k$-form $\omega_{R}$ is invariant under the pull-back of an extended M\"obius transformation and that ${\mathbb R}$ is a $\chi_{k}$-field of definition if such an extended M\"obius transformation can be chosen to be a reflection (Proposition \ref{notita}).

\medskip
\noindent
{\bf Notations.}
\begin{enumerate}[leftmargin=15pt]
\item If $q$ is a prime integer, then we donote by ${\mathbb F}_{q^{r}}$ the finite field with $q^{r}$ elements. 

\item We will denote by ${\mathcal Q}$ the basic subfield of ${\mathbb L}$ (so, if 
$q \geq 0$ is the characteristic of ${\mathbb L}$, then, up to isomorphisms, ${\mathcal Q}={\mathbb Q}$ for $q=0$ and ${\mathcal Q}={\mathbb F}_{q}$ for $q>0$). 

\item If $R \in {\mathbb L}(z)$ is non-zero, then we set $R^{0}=1$.

\item If ${\mathbb K}$ is a subfield of the field ${\mathbb L}$, then we denote by ${\rm G}({\mathbb L}/{\mathbb K})$ the subgroup of ${\rm G}({\mathbb L})$ consisting of those automorphisms fixing each element of ${\mathbb K}$. 
\end{enumerate}

%%%%%%%%%%%%%%%%%%%%
%%%%%%%%%%%%%%%%%%%%
\section{Preliminaries}
In this section, as in the rest of this paper, ${\mathbb L}$ will denote an algebraically closed field of characteristic $q \geq 0$. The interest cases are when ${\mathbb L}$ is either of characteristic zero or
the algebraic closure of a finite field.

%%%%%%%%%%%%%%%%%%%%
\subsection{The finite groups of ${\rm PSL}_{2}({\mathbb L})$}
We denote by $\mathbb{Z}_n$ the cyclic group of order $n$, by $\ D_n$ the dihedral group of order $2n$, by $\ A_4$ (respectively, $\ A_5$) the alternating group of order $12$ (respectively, of order $60$) and by $\ S_4$ the symmetric group of order $24$.

\begin{theo}[\cite{Dickson, VaMa}]\label{teo1}
If $G$ is a finite subgroup of ${\rm PSL}_{2}({\mathbb L})$. Then, $G$ is isomorphic to one of the following groups 
{\small
$$\mathbb{Z}_n,\ D_n,\ A_4,\ S_4,\ A_5:\ (\text{if $q=0$ or if $|G|$ is prime to $q$}),$$ $${\mathbb Z}_q^t,\ {\mathbb Z}_q^t\rtimes {\mathbb Z}_{m},\ {\rm PGL}_{2}({\mathbb F}_{q^{r}}),\  {\rm PSL}_{2}({\mathbb F}_{q^{r}}):\ (\text{if $|G|$ is divisible by $q$}),$$
} where $m$ is a divisor of $q^t-1$. Moreover, the signature of the quotient orbifold ${\mathbb P}^{1}/G$ is given in the Table \ref{Table1}, where $\alpha=\frac{q^r(q^r-1)}{2}$, $\beta=\frac{q^r+1}{2}$.
{\small
\begin{table}[h]
$$\begin{tabular}{|c|c|c|}
\hline
Case & $G$ & signature of ${\mathbb P}^{1}/G$\\
\hline
$1$ & $\mathbb{Z}_n $ & $(n, n)$\\
$2$ & $D_n$ & $(2, 2, n)$\\
$3$ & $A_4$, $q\not=2, 3$ & $(2, 3, 3)$\\
$4$ & $S_4$, $q\not=2, 3$ & $(2, 3, 4)$\\
$5$ & $A_5$, $q\not=2, 3, 5$ & $(2, 3, 5)$\\
$6$  & $A_5$, $q=3$ & $(6,5 )$\\
$7$ & ${\mathbb Z}_q^t$ & $(q^t)$\\
$8$ & ${\mathbb Z}_q^t\rtimes {\mathbb Z}_m$ & $(mq^t, m)$\\
$9$ & ${\rm PSL}_{2}({\mathbb F}_{q^{r}})$, $q\not=2$ & $(\alpha, \beta)$\\
$10$ & ${\rm PGL}_{2}({\mathbb F}_{q^{r}})$ & $(2\alpha, 2\beta)$\\
\hline
\end{tabular}
 $$
 \vspace{0.3cm}
\caption{Groups and signatures in Theorem \ref{teo1}}\label{Table1}
 \end{table}
 }
\end{theo}

%%%%%%%%%%%%%%%%%%
\subsection{Weil's descent theorem}
Let us recall that a smooth projective algebraic variety $Y \subset {\mathbb P}^{n}$ is defined over a subfield ${\mathcal R}$ of ${\mathbb L}$  if its corresponding ideal of polynomials ${\mathbb L}[Y]$ can be generated by a finite collection of polynomials with coefficients in ${\mathcal R}$. We will denote by ${\mathbb L}(Y)$ its field of rational maps defined over ${\mathbb L}$.

Let ${\mathbb K}$ be a perfect subfield of the algebraically closed field ${\mathbb L}$, and set $\Gamma={\rm G}({\mathbb L}/{\mathbb K})$. Assume that $X \subset {\mathbb P}^{n}$ be a given algebraic variety, defined over ${\mathbb L}$.
For each $\sigma \in \Gamma$, we may consider the new algebraic variety $\sigma(X):=X^{\sigma} \subset {\mathbb P}^{n}$. This produces a well-defined action of $\Gamma$ on the set of birational isomorphism classes of algebraic varieties. 

A collection $\{f_{\sigma}:X \to X^{\sigma}\}_{\sigma \in \Gamma}$, of birational isomorphisms defined over ${\mathbb L}$, 
 is called a {\it Weil's datum} for $X$ and ${\mathbb L}/{\mathbb K}$ if it satisfies
(the so-called {\it Weil's co-cycle condition}) $f_{\tau\sigma}=f_{\sigma}^{\tau} \circ f_{\tau} \,\, \mbox{ for each } \,\, \sigma,\tau \in \Gamma.$

Assume $X$ is definable over ${\mathbb K}$, that is, there exists an algebraic variety $Y \subset {\mathbb P}^{m}$, defined over ${\mathbb K}$, and there exists an isomorphism $R:X \to Y$, defined over ${\mathbb L}$. In this case, for each $\sigma \in\Gamma$, we have that 
$Y=Y^{\sigma}$ and that $R^{\sigma}:X^{\sigma} \to Y$ is an isomorphism defined over ${\mathbb K}$, in particular, 
$f_{\sigma}:=(R^{\sigma})^{-1} \circ R: X \to X^{\sigma}$ is is an isomorphism, defined over ${\mathbb L}$. It can be checked that the collection
$\{f_{\sigma}:=(R^{\sigma})^{-1} \circ R\}_{\sigma \in \Gamma}$ is a  Weil's datum for $X$ and ${\mathbb L}/{\mathbb K}$. Weil's descent theorem asserts that this is also a sufficient condition for $X$ to be definable over ${\mathbb K}$.

\begin{theo}[Weil's descent theorem \cite{Weil}]
Let ${\mathbb K}$ be a perfect subfield of the algebraically closed field ${\mathbb L}$. We assume that ${\mathbb L}$ is finitely generated over ${\mathbb K}$. Set $\Gamma={\rm G}({\mathbb L}/{\mathbb K})$.
Let $X$ be an affine algebraic variety defined over ${\mathbb L}$. Then

\begin{enumerate}[leftmargin=15pt]
\item[(a)] If $X$ admits a Weil's datum $\{f_{\sigma}\}_{\sigma \in \Gamma}$ with respect to  ${\mathbb L}/{\mathbb K}$, then there exists an algebraic variety $Y$, defined over ${\mathbb K}$, and there exists a birational isomorphism $R:X \to Y$, defined over ${\mathbb L}$, such that $R=R^{\sigma} \circ f_{\sigma}$ for every $\sigma \in \Gamma$. Moreover, if all the isomorphisms $f_{\sigma}$ are biregular, then $R$ can be chosen to be biregular. 

\item[(b)] If there is another birational isomorphism  $\hat{R}:X \to \hat{Y}$, defined over ${\mathbb L}$, where $\hat{Y}$ is defined over ${\mathbb K}$, such that $\hat{R}=\hat{R}^{\sigma} \circ f_{\sigma}$ for every $\sigma \in \Gamma$, then there exists a birational isomorphism $J:Y \to \hat{Y}$, defined over ${\mathbb K}$, such that $\hat{R}=J \circ R$.
\end{enumerate}
\end{theo}

%%%%%%%%%%%%%%%%%
\subsection{A normalization for rational maps}
We make a normalization of the rational maps as follows. If $R \in {\mathbb L}(z) \setminus \{0\}$, then $R(z)=P(z)/Q(z)$, where $P(z), Q(z) \in {\mathbb L}[z] \setminus \{0\}$ with $P(z)$ and $Q(z)$ being relatively prime polynomials. The choices of $P(z)$ and $Q(z)$ are not unique, but any other choice is obtained by multiplying $P(z)$ and $Q(z)$ by a non-zero number. We normalize our choices so that the leading coefficient of $P$ is $1$ (then $P$ and $Q$ are uniquely determined by $R$). In this case, we say that $R$ is {\it normalized}.

If $\sigma \in {\rm G}({\mathbb L}/{\mathbb K})$ satisfies that $R^{\sigma}=R$, then we know the existence of a non-zero number $\lambda_{\sigma} \in {\mathbb L}$ so that $P^{\sigma}(z)=\lambda_{\sigma} P(z)$ and $Q^{\sigma}(z)=\lambda_{\sigma} Q(z)$. Under the assumption that $R$ is normalized, $\lambda_{\sigma}=1$, that is, each coefficient of $P(z)$ and $Q(z)$ is fixed by $\sigma$.

%%%%%%%%%%%%%%%%%
\subsection{An auxiliary lemma}

\begin{lemm}\label{lemaaux}
Let $B$ be a smooth algebraic curve of genus zero, defined over a subfield ${\mathbb K}$ of ${\mathbb L}$, and assume there is a non-constant ational map $S \in {\mathbb L}(B)$ which is defined over ${\mathbb K}$. Then $B$ has a ${\mathbb K}_{1}$-rational point, where ${\mathbb K}_{1}$ is either equal to ${\mathbb K}$ or it is a quadratic extension of it.
\end{lemm}
\begin{proof}
The divisor $D$ of fixed points of $S$ has degree $d+1$, where $d \geq 1$ is the degree of $S$, and it is ${\mathbb K}$-rational. As $B$ is ${\mathbb K}$-rational, we may find a ${\mathbb K}$-rational canonical divisor ${\mathcal K}$ on $B$. As $B$ has genus zero, ${\mathcal K}$ has degree $-2$. So, by adding to $D$ some positive integer multiple of ${\mathcal K}$, we may find a divisor ${\mathcal E}$ of degree $1$ (when $d$ is even) or $2$ (when $d$ is odd) which is ${\mathbb K}$-rational. 
If ${\mathcal E}=[p]$ or ${\mathcal E}=2[p]$, then $p$ is a ${\mathbb K}$-rational point.
If ${\mathcal E}=[p]+[q]$, $p \neq q$, then we may consider the subgroup $\Gamma_{p}$ of $\Gamma={\rm G}({\mathbb L}/{\mathbb K})$ formed of those $\sigma$ with the property that $\sigma(p)=p$. Either $\Gamma_{p}$ is equal to $\Gamma$, or it is an index two subgroup. If ${\mathbb K}_{1}$ is the fixed field of $\Gamma_{p}$, then $p$ is a ${\mathbb K}_{1}$-rational point.
\end{proof}

%%%%%%%%%%%%%%%%%
%%%%%%%%%%%%%%
\section{Fields of definition and field of moduli of right actions}\label{Sec:actions}
Let us consider a (right) action of ${\rm PSL}_{2}({\mathbb L})$ on ${\mathbb L}(z)$
{\small
$$\chi:{\rm PSL}_{2}({\mathbb L}) \times {\mathbb L}(z) \to {\mathbb L}(z): (T,R) \mapsto \chi(T)(R),$$
}
that is, 
\begin{enumerate}[leftmargin=15pt]
\item[(1)] $\chi(I)(R)=R$, for every $R \in {\mathbb L}(z)$, and 

\item[(2)] $\chi(AB)(R)=\chi(B)(\chi(A)(R))$, for every $A,B \in {\rm PSL}_{2}({\mathbb L})$ and every $R \in {\mathbb L}(z)$. 
\end{enumerate}

Two rational maps 
$R,S \in {\mathbb L}(z)$ are called {\it $\chi$-equivalent} if there exists $T \in {\rm PSL}_{2}({\mathbb L})$ with $S=\chi(T)(R)$; we denote this shortly by $S\equiv_{\chi} R$.

%%%%%%%%%%%%%%%%
\subsection{The ${\rm \bf G}({\mathbb L})${\bf -property}}
In general, if $S\equiv_{\chi} R$ and $\sigma \in {\rm G}({\mathbb L})$, then it is not clear that $R^{\sigma}\equiv_{\chi}S^{\sigma}$. To avoid such a problem, we assume the following ${\rm G}({\mathbb L})$-compatibility property.

\medskip
\noindent
${\rm \bf G}({\mathbb L})${\bf -property:}
{\it The action $\chi$ is said to satisfy the ${\rm G}({\mathbb L})$-property if: for every $\sigma \in {\rm G}({\mathbb L})$, every $T \in {\rm PSL}_{2}({\mathbb L})$ and every $R \in {\mathbb L}(z)$, it holds that 
$\left( \chi(T)(R)\right)^{\sigma}=\chi(T^{\sigma})(R^{\sigma}).$
}

\begin{rema}
The ${\rm G}({\mathbb L})$-property
ensures that, if $R \equiv_{\chi} S$ and 
$\sigma \in {\rm G}({\mathbb L})$, then $R^{\sigma} \equiv_{\chi} S^{\sigma}$.
\end{rema}

%%%%%%%%%%%%%%
\subsection{The $\chi$-fields of definition}
A {\it $\chi$-field of definition} of $R \in {\mathbb L}(z)$ (or just a field of definition if the $\chi$-action is clear) is a subfield ${\mathbb K}$ of ${\mathbb L}$ such that there is some $S \in {\mathbb K}(z)$ with $S\equiv_{\chi} R$, that is, if there is some $T \in {\rm PSL}_{2}({\mathbb L})$ such that $S=\chi(T)(R)$.
If, moreover, $\chi$ satisfies the ${\rm G}({\mathbb L})$-property, then 
for every  $\sigma \in {\rm G}({\mathbb L})$, it holds that
{\small
$$\chi(T^{\sigma})(R^{\sigma})=\left( \chi(T)(R)\right)^{\sigma}=S^{\sigma}=S=\chi(T)(R).$$
}

\begin{rema}
Note that, under the ${\rm G}({\mathbb L})$-property, $\chi$-equivalent rational maps have the same $\chi$-fields of definition.
\end{rema}

%%%%%%%%%%%%%%
\subsection{The $\chi$-field of moduli}
If $R \in {\mathbb L}(z)$, then we set 
{\small
$$
U_{R}^{\chi}:=\{ \sigma \in {\rm G}({\mathbb L}): R^{\sigma} \equiv_{\chi} R\} \subset {\rm G}({\mathbb L}).
$$
}

\begin{lemm}\label{lema21}
Let $R \in {\mathbb L}(z)$.
If $\chi$ satisfies the ${\rm G}({\mathbb L})$-property, then 
$U_{R}^{\chi}$
is a subgroup of ${\rm G}({\mathbb L})$. Moreover, if $R \equiv_{\chi} S$, then $U_{R}^{\chi}=U_{S}^{\chi}$.
\end{lemm}
\begin{proof}
If $e \in {\rm G}({\mathbb L})$ denotes the identity, then $R^{e}=R$, so $e \in U_{R}^{\chi}$. If $\sigma \in U_{R}^{\chi}$, then $R^{\sigma} \equiv_{\chi} R$ asserts that 
$R \equiv_{\chi} R^{\sigma^{-1}}$, so $\sigma^{-1} \in U_{R}^{\chi}$.
If $\sigma, \tau \in U_{R}^{\chi}$, then $R^{\tau} \equiv_{\chi} R$, so by the ${\rm G}({\mathbb L})$-property, we have $R^{\sigma\tau} \equiv_{\chi} R^{\sigma} \equiv_{\chi} R$, so $\sigma \tau \in U_{R}^{\chi}$. 
For the last part, if $\sigma \in U_{R}^{\chi}$, then
$S \equiv_{\chi} R \equiv_{\chi} R^{\sigma} \equiv_{\chi} S^{\sigma}$, i.e., $\sigma \in U_{S}^{\chi}$
\end{proof}
 
As a consequence of Lemma \ref{lema21}, if $\chi$ satisfies the ${\rm G}({\mathbb L})$-property and $R \in {\mathbb L}(z)$, then 
$U_{R}^{\chi}$ is a subgroup of ${\rm G}({\mathbb L})$. The fixed field 
{\small
$$
{\mathcal M}_{R}^{\chi}={\rm Fix}(U_{R}^{\chi})<{\mathbb L}
$$
}
is  called the {\it $\chi$-field of moduli} (or just the field of moduli if the $\chi$-action is clear) of $R$.

\begin{rema} If $\chi$ satisfies the ${\rm G}({\mathbb L})$-property, then the following facts hold.
(1) $\chi$-equivalent rational maps have the same $\chi$-field of moduli. 
(2) Every $\chi$-field of definition of $R$ contains ${\mathcal M}_{R}^{\chi}$. In fact, if  ${\mathbb K} \leq {\mathbb L}$, $R \in {\mathbb K}(z)$ and $\sigma \in {\rm G}({\mathbb L})$, then $R^{\sigma}=R$ and  $\sigma \in U_{R}^{\chi}$. It might be that the $\chi$-field of moduli ${\mathcal M}_{R}^{\chi}$ is not a $\chi$-field of definition. 
\end{rema}

%%%%%%%%%%%%%%%%%
\subsection{$\chi$-cocycles and $\chi$-fields of definitions}
Necessary and sufficient conditions for a subfield ${\mathbb K}$ of ${\mathbb L}$ to be a $\chi$-field of definition of a rational map $R \in {\mathbb L}(z)$ can be described in terms of certain cocycles.
A $\chi$-cocycle, for $R$ with respect to the extension ${\mathbb L}/{\mathbb K}$, is a function 
{\small
$$
C:{\rm G}({\mathbb L}/{\mathbb K}) \to {\rm PSL}_{2}({\mathbb L}): \sigma \mapsto C(\sigma)=T_{\sigma},
$$
}
such that:
\begin{enumerate}[leftmargin=15pt]
\item[(i)] $R^{\sigma}=\chi(T_{\sigma})(R),$ for every $\sigma \in {\rm G}({\mathbb L}/{\mathbb K})$, and
\item[(ii)] $T_{\sigma\tau}=T_{\sigma} \circ T_{\tau}^{\sigma},$ for every $\sigma, \tau \in {\rm G}({\mathbb L}/{\mathbb K}).$
\end{enumerate}

Fort short, we will identify the above $\chi$-cocycle with the collection $\{T_{\sigma}\}_{\sigma \in {\rm G}({\mathbb L}/{\mathbb K})}$. Note that the collection $\{f_{\sigma}=T_{\sigma}^{-1}\}_{\sigma \in {\rm G}({\mathbb L}/{\mathbb K})}$ determines a Weil's datum.

The above $\chi$-cocycle is called a $\chi$-coboundary if there is some $T \in {\rm PSL}_{2}({\mathbb L})$ such that $T_{\sigma}=T^{-1} \circ T^{\sigma}$. The following asserts that ${\mathbb K}$ is a $\chi$-field of definition of $R$ if and only if there is $\chi$-coboundary for the extension ${\mathbb L}/{\mathbb K}$.

\begin{theo}\label{TeoWeil0}
Let $\chi$ be a (right) action on ${\mathbb L}(z)$ satisfying the ${\rm G}({\mathbb L})$-property.
A subfield ${\mathbb K}$ of ${\mathbb L}$ is a $\chi$-field of definition of $R$ if and only if 
there exists $\chi$-coboundary for $R$ with respect to the extension ${\mathbb L}/{\mathbb K}$. 
\end{theo}
\begin{proof}
We assume $R$ to be normalized.
Let us first assume that $R$ is $\chi$-definable over ${\mathbb K}$. It follows, from the definition of $\chi$-field of definition of $R$, the existence of a M\"obius transformation $T \in {\rm PSL}_{2}({\mathbb L})$ and some $S \in {\mathbb K}(z)$ so that $R=\chi(T)(S)$. If $\sigma \in {\rm G}({\mathbb L}/{\mathbb K})$, then $S^{\sigma}=S$ and, for $T_{\sigma}=T^{-1} \circ T^{\sigma}$, it holds that 
{\small
$$R^{\sigma}=(\chi(T)(S))^{\sigma}=\chi(T^{\sigma})(S^{\sigma})=\chi(T^{\sigma})(S)=\chi(T \circ T_{\sigma})(S)=\chi(T_{\sigma})(\chi(T)(S))=\chi(T_{\sigma})(R).$$
}

Conversely, assume there exists $T \in {\rm PSL}_{2}({\mathbb L})$ such that, for every $\sigma \in {\rm G}({\mathbb L}/{\mathbb K})$, it holds the equality 
$R^{\sigma}=\chi(T_{\sigma})(R),$ where $T_{\sigma}=T^{-1} \circ T^{\sigma}$. If $S=\chi(T^{-1})(R)$, then $S \equiv_{\chi} R$.
If $\sigma \in {\rm G}({\mathbb L}/{\mathbb K})$, then
{\small
$$S^{\sigma}=(\chi(T^{-1})(R))^{\sigma}=\chi((T^{\sigma})^{-1})(R^{\sigma})=\chi((T \circ T_{\sigma})^{-1})(R^{\sigma})=$$
$$
=\chi(T_{\sigma}^{-1} \circ T^{-1})(R^{\sigma})=\chi(T^{-1})(\chi(T_{\sigma}^{-1})(R^{\sigma}))=\chi(T^{-1})(R)=S.$$
}

Now, the above asserts that $S^{\sigma}=S$, for every $\sigma \in  {\rm G}({\mathbb L}/{\mathbb K})$, in particular, by our normalization, $S \in {\mathbb K}(z)$.
\end{proof}

%%%%%%%%%%%%%%%%
\subsection{The $\chi$-automorphism group of a rational map}
If $R \in {\mathbb L}(z)$, then the $\chi$-group of automorphisms is 
the following subgroup of ${\rm PSL}_{2}({\mathbb L})$
{\small
$$
{\rm Aut}_{\chi}(R)=\{ T \in {\rm PSL}_{2}({\mathbb L}): \chi(T)(R)=R\}.
$$
}

In many cases, for $R$ of degree at least two, the above group is finite (for instance, when $\chi \in \{\chi_{k}: k \in {\mathbb Z} \cup \{\infty\}\}$).

%%%%%%%%%%%%%%%%%%
\subsection{Existence of $\chi$-cocycles}
If $R \in {\mathbb L}(z)$ (say in normalized form) and ${\mathbb K}$ is a subfield of ${\mathbb L}$, then we denote by ${\mathbb L}_{R,{\mathbb K}} \leq {\mathbb L}$ the algebraic closure of the subfield generated by ${\mathbb K}$ and the coefficients of $R$. This algebraically closed field ${\mathbb L}_{R,{\mathbb K}}$ is a finite degree extension of ${\mathbb K}$.

\begin{lemm}\label{finitedegreeextension}
Let $\chi$ be a (right) action on ${\mathbb L}(z)$ satisfying the ${\rm G}({\mathbb L})$-property and ${\mathbb K}$ be a perfect subfield of ${\mathbb L}$. 
Let $R \in {\mathbb L}(z)$ (say to be normalized) be such that ${\rm Aut}_{\chi}(R)$ is finite. 
Assume that, for every $\sigma \in {\rm G}({\mathbb L}/{\mathbb K})$, there is a $\hat{T}_{\sigma} \in {\rm PSL}_{2}({\mathbb L})$ such that 
$R^{\sigma}=\chi(T_{\sigma})(R)$. 
Then, for every $\sigma \in {\mathbb L}/{\mathbb K}$, the transformation $\hat{T}_{\sigma}$ is defined over ${\mathbb L}_{R,{\mathbb K}}$. 
\end{lemm}
\begin{proof}
Let $\sigma \in {\mathbb L}/{\mathbb K}$. For every $\tau \in {\rm G}({\mathbb L}/{\mathbb L}_{R,{\mathbb K}})$, we note that $R^{\tau}=R$, 
$(R^{\sigma})^{\tau}=R^{\sigma}$, and 
$T_{\sigma}^{-1} \circ T_{\sigma}^{\tau} \in {\rm Aut}_{\chi}(R)$. As ${\rm Aut}_{\chi}(R)$ is finite, the above asserts that $T_{\sigma}$ must be defined over a finite extension of ${\mathbb L}_{R,{\mathbb K}}$, so over it.
\end{proof}

\begin{rema}
The above lemma asserts that, under the assumption that ${\rm Aut}_{\chi}(R)$ is finite, even if ${\mathbb L}$ is not finitely generated over the lower field ${\mathbb K}$, we may replace ${\mathbb L}$ with ${\mathbb L}_{R,{\mathbb K}}$ to apply Weil's descent theorem.
 This observation will be used frequently each time we apply Weil's descent theorem.
\end{rema}

The following result asserts the existence of $\chi$-cocycles in some cases. 

\begin{theo}\label{teocociclo}
Let $\chi$ be a (right) action on ${\mathbb L}(z)$ satisfying the ${\rm G}({\mathbb L})$-property.
Let ${\mathbb K}$ be a perfect subfield of ${\mathbb L}$. 
Let $R \in {\mathbb L}(z)$ be such that its group of $\chi$-automorphisms ${\rm Aut}_{\chi}(R)$
is finite. Assume that, for every $\sigma \in {\rm G}({\mathbb L}/{\mathbb K})$, there is a $\hat{T}_{\sigma} \in {\rm PSL}_{2}({\mathbb L})$ such that 
$R^{\sigma}=\chi(T_{\sigma})(R)$. Then
\begin{enumerate}
\item If ${\rm Aut}_{\chi}(R)$ is not a non-trivial cyclic group, then there is a $\chi$-cocycle for $R$ with respect to the extension ${\mathbb L}/{\mathbb K}$.

\item If ${\rm Aut}_{\chi}(R)$ is a non-trivial cyclic group, then there is subfield ${\mathbb K}_{1}$ of ${\mathbb L}$ which is either equal to ${\mathbb K}$ or a quadratic extension of it, and there is a $\chi$-cocycle for $R$ with respect to the extension ${\mathbb L}/{\mathbb K}_{1}$.
\end{enumerate}
\end{theo}
\begin{proof}
By Lemma \ref{finitedegreeextension}, we may assume that ${\mathbb L}$ is a finitely generated extension of ${\mathbb K}$.

(1) Let us first assume that ${\rm Aut}_{\chi}(R)$ is not a non-trivial cyclic group.
If ${\rm Aut}_{\chi}(R)$ is the trivial group, then the provided collection $\{\hat{T}_{\sigma}\}_{\sigma \in {\rm G}({\mathbb L}/{\mathbb K})}$ is already a $\chi$-cocycle as desired. So, let us assume, from now on, that ${\rm Aut}_{\chi}(R)$ is non-trivial (as listed in Theorem \ref{teo1}).
Let us consider a branched regular cover $\pi:{\mathbb P}^{1} \to {\mathbb P}^{1}$, with deck group ${\rm Aut}_{\chi}(R)$, whose branch values are contained in the set $\{\infty,0,1\}$. Note that, for each $\sigma \in {\rm G}({\mathbb L}/{\mathbb K})$, the branch values of $\pi^{\sigma}$ are again contained inside $\{\infty,0,1\}$.
In cases (2)--(5), in Table \ref{Table1}, the three values are the branch values of $\pi$.  If  ${\rm Aut}_{\chi}(R)$ is not isomorphic to ${\mathbb Z}_{2}^{2}$, one of these branch values, say $\infty$,  has branch order diffrent from the others two.
In cases (6), (8), (9), and (10), these branch values are assumed to be $\infty$ and $0$. Both of them have different branch orders.
In case (7), we assume the branch value to be $\infty$.

By our hypothesis, for each $\sigma \in {\rm G}({\mathbb L}/{\mathbb K})$, there is a $\hat{T}_{\sigma} \in {\rm PSL}_{2}({\mathbb L})$ satisfying that 
$R^{\sigma}=\chi(\hat{T}_{\sigma})(R)$.  Also, (i) ${\rm Aut}(R^{\sigma})=\hat{T}_{\sigma}^{-1} {\rm Aut}(R) \hat{T}_{\sigma}$,  
(ii) $\pi^{\sigma}:{\mathbb P}^{1} \to {\mathbb P}^{1}$ is a branched regular cover with deck group ${\rm Aut}(R^{\sigma})$, and 
(iii) there is $S_{\sigma} \in {\rm PSL}_{2}({\mathbb L})$ such that $S_{\sigma} \circ \pi^{\sigma}=\pi \circ \hat{T}_{\sigma}$. 
The transformation $S_{\sigma}$ is uniquely determined by $\sigma$ (it does not depend on the chosen $\hat{T}_{\sigma}$ because any other choice for $\hat{T}_{\sigma}$ is obtained by composition at the left with an element of ${\rm Aut}(R)$). This uniqueness ensures that 
the collection $\{S_{\sigma}^{-1}\}_{\sigma \in {\rm G}({\mathbb L}/{\mathbb K})}$ satisfies Weil's co-cycle conditions \cite{Weil} (the condition for ${\mathbb K}$ to be perfect is needed in here) with respect to the field ${\mathbb K}$, that is, there is a smooth projective curve of genus zero, say $\hat{B}$, defined over ${\mathbb K}$, and an isomorphism $\hat{L}:{\mathbb P}^{1} \to \hat{B}$ so that, for every $\sigma \in {\rm G}({\mathbb L}/{\mathbb K})$, it holds that $\hat{L}=\hat{L}^{\sigma} \circ S_{\sigma}^{-1}$. 

Let us consider the following divisor $D \in {\rm Div}(\hat{B})$:
\begin{enumerate}[leftmargin=15pt]

\item $D=[\hat{L}(0)]+[\hat{L}(1)]+[\hat{L}(\infty)]$, if ${\rm Aut}_{\chi}(R)$ is isomorphic to ${\mathbb Z}_{2}^{2}$;

\item $D=[\hat{L}(\infty)]$, if we are in the rest of the cases.

\end{enumerate}

As each $S_{\sigma}$ must preserve the branch divisor of $\pi$ (which is the same as for $\pi^{\sigma}$) and their branch orders, it may be observed that $D$ is ${\mathbb K}$-rational. So, 
in the case (2),  the point $\hat{L}(\infty)$ is a ${\mathbb K}$-rational point in $\hat{B}$. 
In the case (1), 
by adding to $D$ a canonical ${\mathbb K}$-divisor of $\hat{B}$ (which has degree $-2$), we may find an effective  ${\mathbb K}$-divisor of degree $1$.  As $\hat{B}$ is defined over ${\mathbb K}$ and it has genus zero, the Riemann-Roch ensures the existence of a  ${\mathbb K}$-rational point in $\hat{B}$.

As $\hat{B}$ has genus zero, the existence of one ${\mathbb K}$-rational point asserts that $\hat{B}$ has infinitely many ${\mathbb K}$-rational points. 
In particular, we may take one of these points, say $b$, which is not a branch value of $\hat{L} \circ \pi:{\mathbb P}^{1} \to \hat{B}$. 
Take a point $p \in {\mathbb P}^{1}$ so that $\hat{L}(\pi(p))=b$. We may see that $\sigma(p)$ and $\hat{T}_{\sigma}^{-1}(p)$ are projected to the same point under $\pi^{\sigma}$. It follows the existence of a unique $h_{\sigma} \in {\rm Aut}(R^{\sigma})$ with $h_{\sigma}(\hat{T}_{\sigma}^{-1}(p))=\sigma(p)$. 
If we consider $T_{\sigma}=\hat{T}_{\sigma} \circ h_{\sigma}^{-1}$, then we obtain a $\chi$-coclycle for $R$ with respect to the extension ${\mathbb L}/{\mathbb K}$.

(2) Now, assume ${\rm Aut}_{\chi}(R)$ is a non-trivial cyclic group, say generated by an order $m \geq 2$ element $A$. Let $p, q \in {\mathbb P}^{1}$ be the two fixed points of $A$. By changing $R$ by a suitable $\chi$-equivalent one, we may assume $A(z)=e^{2 \pi i/m}$; so $p=\infty$ and $q=0$.
We may proceed as before, but in this case, we consider the branch divisor of $\hat{L} \circ \pi:{\mathbb P}^{1} \to B$, i.e., $D=[\hat{p}]+[\hat{q}]$, where $\hat{p}=\hat{L}(\pi(p))$ and $\hat{q}=\hat{L}(\pi(q))$. This divisor is of dregree two and ${\mathbb K}$-rational. Now, each of the points $\hat{p}$ and $\hat{q}$ 
is either ${\mathbb K}$-rational or ${\mathbb K}_{1}$-rational for a suitable quadratic extension of ${\mathbb K}$. Now, we follow as above using ${\mathbb K}_{1}$ instead of ${\mathbb K}$.
\end{proof}

\begin{rema}\label{observabien}
In the above proof, part (2), we have obtained an at most quadratic extension ${\mathbb K}_{1}$ of ${\mathbb K}$. The $\chi$-cocycle 
$\{T_{\sigma}\}_{\sigma \in {\rm G}({\mathbb L}/{\mathbb K}_{1})}$ produces a Weil's datum $\{f_{\sigma}=T_{\sigma}^{-1}\}_{\sigma \in {\rm G}({\mathbb L}/{\mathbb K}_{1})}$. By Weil's descent theorem, we obtain a genus zero smooth algebraic curve $B$, defined over ${\mathbb K}_{1}$, and an isomorphism $L:{\mathbb P}^{1} \to B$, such that $L=L^{\sigma} \circ f_{\sigma}$. The point $L(p)$ is a ${\mathbb K}_{1}$-rational point. In fact, for each $\sigma \in {\rm G}({\mathbb L}/{\mathbb K}_{1})$, 
$L(p)^{\sigma}=L^{\sigma}(p^{\sigma})=L \circ T_{\sigma} (\sigma(p))=L(p).$
\end{rema}

\begin{coro}\label{coroaux}
Let $\chi$ be a (right) action on ${\mathbb L}(z)$ satisfying the ${\rm G}({\mathbb L})$-property.
Let $R \in {\mathbb L}(z)$ be such that ${\rm Aut}_{\chi}(R)$ is finite and ${\mathcal M}_{R}^{\chi}$ is a perfect subfield of ${\mathbb L}$.
\begin{enumerate}
\item If ${\rm Aut}_{\chi}(R)$ is not a non-trivial cyclic group, then there is a $\chi$-cocycle for $R$ with respect to the extension ${\mathbb L}/{\mathcal M}_{R}^{\chi}$.
\item If ${\rm Aut}_{\chi}(R)$ is a non-trivial cyclic group, then there is an (at most quadratic) extension ${\mathbb K}_{1}$ of ${\mathcal M}_{R}^{\chi}$ and there is a 
$\chi$-cocycle for $R$ with respect to the extension ${\mathbb L}/{\mathbb K}_{1}$.
\end{enumerate}
\end{coro}

%%%%%%%%%%%%%%
%%%%%%%%%%%%%%
\section{Universal right actions}\label{Sec:universal}
Let us consider a (right) action $\chi:{\rm PSL}_{2}({\mathbb L}) \times {\mathbb L}(z) \to {\mathbb L}(z)$, which datisfies the ${\rm G}({\mathbb L})$-property.

We may think of $\chi$ as a cofunctor that associate to every $T \in {\rm PSL}_{2}({\mathbb L})$ a bijection 
{\small
$$\chi(T):{\mathbb L}(z) \to {\mathbb L}(z): R \mapsto \chi(T)(R).$$
}

We will say that  the action $\chi$  is a {\it universal right action} if it can be extended, as a cofunctor, as follows:
\begin{enumerate}
\item if $X$ and $Y$ are smooth algebraic curves of genus zero, both defined over ${\mathbb L}$, and $\Psi:X \to Y$ is an 
isomorphism, also defined over ${\mathbb L}$, then $\chi(\Psi):{\mathbb L}(Y) \to {\mathbb L}(X)$ is a bijective map, where ${\mathbb L}(X)$ and ${\mathbb L}(Y)$ denote the spaces of ${\mathbb L}$-rational maps on $X$ and $Y$, respectively;

\item for every $\sigma \in {\rm G}({\mathbb L}/{\mathbb K})$ and $R \in {\mathbb L}(Y)$, it holds that $\chi(\Psi)(R)^{\sigma}=\chi(\Psi^{\sigma})(R^{\sigma})$;

\item $\chi(I)=I$;

\item if $X,Y,Z$ are smooth algebraic curves of genus zero, each defined over ${\mathbb L}$, and 
$\Psi_{2}:X \to Y$ and $\Psi_{1}:Y \to Z$ are isomorphisms, both defined over ${\mathbb L}$, then 
$\chi(\Psi_{1} \circ \Psi_{2})=\chi(\Psi_{2}) \circ \chi(\Psi_{1})$.
\end{enumerate}

Examples of such kinds of actions are $\chi=\chi_{\infty}$ and $\chi_{k}$ (as this comes from the pull-back of rational  $k$-forms) which are discussed in the next sections.

%%%%%%%%%%%%%%%%
\subsection{Sometimes the existence of a $\chi$-cocycle provides the existence of a $\chi$-coboundary}
As previously seen, the subfield ${\mathbb K}$ of ${\mathbb L}$ is a $\chi$-field of definition if we know the existence of a $\chi$-coboundary. Unfortunately, in general, this is a difficult task. On the other hand, to search for just a $\chi$-coclycle is a much easier problem. Sometimes, the existence of a $\chi$-coclycle asserts the existence of a $\chi$-coboundary up to a quadratic extension of ${\mathbb K}$. 

\begin{theo}\label{teomain}
Let ${\mathbb K}$ be a perfect subfield of ${\mathbb L}$, and  $\chi$ be a universal right action.
Assume there is a $\chi$-cocycle $\{T_{\sigma}\}_{\sigma \in {\rm G}({\mathbb L}/{\mathbb K})}$, for $R \in {\mathbb L}(z)$ with respect to the extension ${\mathbb L}/{\mathbb K}$.
Then there exists a subfield ${\mathbb K}_{1}$ of ${\mathbb L}$, either equal to ${\mathbb K}$ or a quadratic extension of it, such that ${\mathbb K}_{1}$ is a $\chi$-field of definition of $R$.
\end{theo}
\begin{proof}
The collection $\{f_{\sigma}=T_{\sigma}^{-1}\}_{\sigma \in {\rm G}({\mathbb L}/{\mathbb K})}$ provides a Weil's datum. So, by Weil's descent theorem, there exists a genus zero smooth algebraic curve $B$, defined over ${\mathbb K}$, and there exists an isomorphism $L:{\mathbb P}^{1} \to B$, defined over ${\mathbb L}$, such that, for every $\sigma \in {\rm G}({\mathbb L}/{\mathbb K})$, it holds that $L=L^{\sigma} \circ T_{\sigma}^{-1}$.

If $S \in {\mathbb L}(B)$ is such that $R=\chi(L)(S)$, then, for every $\sigma \in {\rm G}({\mathbb L}/{\mathbb K})$, 
{\small
$$S=\chi(L^{-1})(R)=\chi(T_{\sigma} \circ (L^{\sigma})^{-1})(R)=\chi((L^{\sigma})^{-1}) (\chi(T_{\sigma})(R))=$$
$$=\chi((L^{\sigma})^{-1})(R^{\sigma})=\chi((L^{-1})^{\sigma})(R^{\sigma})=\chi(L^{-1})(R)^{\sigma}=S^{\sigma},
$$
}
that is, $S \in {\mathbb K}(B)$.

Now, let $b \in B$ be a ${\mathbb K}_{1}$-rational, where ${\mathbb K}_{1}$ is at most a quadratic extension of ${\mathbb K}$ (the existence is guarantee by 
Lemma \ref{lemaaux}). In this case, as a consequence of the Riemann-Roch theorem, there is an isomorphism $\Phi:B \to {\mathbb P}^{1}$, defined over ${\mathbb K}_{1}$. If we consider $\hat{R} \in {\mathbb L}(z)$ such that $\chi(\hat{R})=S$, then one may observe that $\hat{R} \in {\mathbb K}_{1}(z)$. 
As $T=\Phi \circ L \in {\rm PSL}_{2}({\mathbb L})$ and
{\small
$$\chi(T)(\hat{R})=\chi(\Phi \circ L)(\hat{R})=\chi(L)(\chi(\Phi)(\hat{R}))=\chi(L)(S)=R,$$
}
we obtain that ${\mathbb K}_{1}$ is a $\chi$-field of definition of $R$.
\end{proof}

The following result provides an affirmative aswer for the second question in the introduction.

\begin{theo}\label{coromain}
Let $\chi$ be a universal right action and $R \in {\mathbb L}(z)$, with finite ${\rm Aut}_{\chi}(R)$.
If ${\mathcal M}_{R}^{\chi}$ is a perfect field, then ${\rm FOD/FOM}_{\chi}(R) \leq 2$. In particular, this holds if 
${\mathbb L}$ either has characteristic zero or it is the algebraic closure of a finite field.
\end{theo}
\begin{proof}
Let us consider ${\mathbb K}={\mathcal M}_{R}^{\chi}$ in Theorem \ref{teomain}.
If ${\rm Aut}_{\chi}(R)$ is not a non-trivial cyclic group, the result follows from Theorem \ref{teomain} and Corollary \ref{coroaux}.

Let us now assume ${\rm Aut}_{\chi}(R)$ is a non-trivial cyclic group. In this case, the generator of ${\rm Aut}_{\chi}(R)$ has two fixed points. Let $z \in {\mathbb P}^{1}$ be one of them (up to conjugation, we may assume $z=\infty$). In this case, we consider the quadratic extension ${\mathbb K}_{1}$ of ${\mathcal M}_{R}^{\chi}$ obtained in the proof of part (2) of Corollary \ref{coroaux}. 
If $L:{\mathbb P}^{1} \to B$ is as in the proof of Theorem \ref{teomain}, then $b=L(z)$ is a ${\mathbb K}_{1}$-rational point (Remark \ref{observabien}). Now, the result again follows from the final part of the proof of Theorem \ref{teomain}.
\end{proof}

%%%%%%%%%%%%%%%%%
\section{Example 1: The $\chi_{\infty}$-action}
Let us consider the conjugation action $\chi_{\infty}$, which is an example of a universal (right) action and whch satisfies the ${\rm G}({\mathbb L})$-property. Theorem \ref{TeoWeil0}, asserts that a subfield ${\mathbb K}$ of ${\mathbb L}$ is a $\chi_{\infty}$-field of definition of $R$ if and only if there is some $T \in {\rm PSL}_{2}({\mathbb L})$ such that
$R^{\sigma}=\left(T^{\sigma}\right)^{-1} \circ T \circ R \circ T^{-1} \circ T^{\sigma}.$
In this case, 
{\small
$$U_{R}^{\chi_{\infty}}=\left\{ \sigma \in {\rm G}({\mathbb L}): R^{\sigma}=T_{\sigma}^{-1} \circ  R \circ T_{\sigma}, \; {\rm for \; a \; suitable } \; T_{\sigma} \in {\rm PSL}_{2}({\mathbb L}) \right\},$$
}
and
{\small
$${\rm Aut}_{\chi_{\infty}}(R)=\{T \in {\rm PSL}_{2}({\mathbb L}): T^{-1} \circ R \circ T=R\}.$$
}

It is not difficult to check that a rational map of degree at most one (so, with a non-finite group of $\chi_{\infty}$-automorphisms) can be defined over at most a quadratic extension of its $\chi_{\infty}$-field of moduli. If the rational map $R$ has degree bigger or equal to two, then its group of $\chi_{\infty}$-automorphisms is finite. 

The above, together with Theorem \ref{coromain}, gives us the following generalization (to positive characteristic) of the results in \cite{Hidalgo}.

\begin{coro}
Let $R \in {\mathbb L}(z)$.
 If ${\mathcal M}_{R}^{\chi_{\infty}}$ is perfect, then ${\rm FOD/FOM}_{\chi_{\infty}}(R) \leq 2$. In particular, this holds if 
 either ${\mathbb L}$ has characteristic zero or it is the algebraic closure of a finite field.
\end{coro}

\begin{example}\label{ejemplo0}
Let us consider ${\mathbb L}={\mathbb C}$.
For each $\lambda \in {\mathbb C} \setminus \{0\}$, set $R_{\lambda}(z)=z^{2}+\lambda$. As $R_{\lambda}$ has exactly two critical points, $0$ and $\infty$, and $\infty$ is (super)-attracting fixed point, it can be checked that $R_{\lambda} \equiv_{\chi_{\infty}} R_{\mu}$ if and only if $\lambda=\mu$. This permits us to see that 
$U_{R_{\lambda}}^{\chi_{\infty}}=\{\sigma \in {\rm G}({\mathbb C}): \sigma(\lambda)=\lambda\}$, in particular, that 
${\mathcal M}_{R_{\lambda}}^{\chi_{\infty}}={\mathbb Q}(\lambda)$ is a field of definition of $R_{\lambda}$.
\end{example}

%%%%%%%%%%%%%%%%%%%%%
\section{Example 2: The $\chi_{k}$-actions}\label{Sec1}
In this section, we consider the universal right actions $\{\chi_{k}\}_{k \in {\mathbb Z}}$, each one satisfying the ${\rm G}({\mathbb L})$-property.

Let $R \in {\mathbb L}(z)$ and $k \in {\mathbb Z}$. Let us consider the rational $k$-form $\omega_{R} =R(z) dz^{k}$. If $T \in {\rm PSL}_{2}({\mathbb L})$, then we may consider the 
pull-back of $\omega_{R}$ under $T$:
{\small
$$\omega_{R}=R(z) dz^{k} \mapsto  T^{*} \omega_{R}=R(T(z))\left(T'(z)\right)^{k} dz^{k}.$$
}

The above induces the (right) action 
{\small
$\chi_{k}(T)(R):=(R \circ T) \left(T'\right)^{k}.$
}
In this case, the $\chi_{k}$-moduli space of $R$ is ${\mathcal M}_{R}^{\chi_{k}}={\rm Fix}(U_{R}^{\chi_{k}})$, where
{\small
$$U_{R}^{\chi_{k}}=\left\{ \sigma \in {\rm G}({\mathbb L}): R^{\sigma}=(R \circ T_{\sigma})\left(T'_{\sigma}\right)^{k}, \; {\rm for \; a \; suitable } \; T_{\sigma} \in {\rm PSL}_{2}({\mathbb L}) \right\}.$$
}

In particular, 
{\small
$${\rm Aut}_{\chi_{k}}(R)=\{T \in {\rm PSL}_{2}({\mathbb L}): (R \circ T) \left(T'\right)^{k}=R\}.$$
}

\begin{rema}
Note that $T \in {\rm Aut}_{\chi_{k}}(R)$ if and only if $T^{*} \omega_{R} =\omega_{R}$; we also say that $T$ is a {\it holomorphic automorphism} of $\omega_{R}$. We will also say that ${\mathcal M}_{R}^{\chi_{k}}$ is the field of moduli of the meromorphic $k$-form $\omega_{R}$.
\end{rema}

Theorem \ref{TeoWeil0} reads as follows.

\begin{coro}\label{TeoWeil1}
Let $k \in {\mathbb Z}$ be fixed and $R \in {\mathbb L}(z)$. A subfield ${\mathbb K}$ of ${\mathbb L}$ 
 is a $\chi_{k}$-field of definition of $R$ if and only if 
there exists a M\"obius transformation $T \in {\rm PSL}_{2}({\mathbb L})$ (called a Weil's datum for $\omega_{R}$) such that, for every 
$\sigma \in {\rm G}({\mathbb L}/{\mathbb K})$ it holds the following equality
{\small
$$R^{\sigma}(z)=R(T_{\sigma}(z)) \left(T'_{\sigma}(z)\right)^{k}, \quad \mbox{where} \; T_{\sigma}=T^{-1} \circ T^{\sigma}.$$
}
\end{coro}

\begin{rema}
(1)  If $k \in {\mathbb Z}$, $T \in {\rm PSL}_{2}({\mathbb L})$, $R,S \in {\mathbb L}(z)$ and $\lambda \in {\mathbb L}$, it holds that
$\chi_{k}(T,\lambda R+S)=\lambda \chi_{k}(T,R)+\chi_{k}(T,S)$. 
(2) If $k=0$,  then $\chi_{0}(T,R)=R \circ T$. 
The $\chi_{0}$-action, for ${\mathbb L}={\mathbb C}$, restricted to Belyi rational maps (rational maps whose branch values are contained in $\{\infty,0,1\}$) corresponds to an action over genus zero dessin d'enfants. These objects have been an object of study by various authors, and it is related to getting information on the structure of the absolute Galois group ${\rm G}(\overline{\mathbb Q})={\rm Gal}(\overline{\mathbb Q}/{\mathbb Q})$.
\end{rema}

\begin{rema}\label{negaposi}
There is a natural isomorphism between the space of rational $k$-forms and the space of rational $(-k)$-forms defined by the rule:
{\small
$\omega_{R}=R dz^{k} \mapsto \Theta(\omega_{R})=\frac{1}{R}dz^{-k}.$
}
Moreover, if $T \in {\rm PSL}_{2}({\mathbb L})$ and $\sigma \in {\rm G}({\mathbb L})$, then $T^{*}\Theta(\omega_{R})=\Theta(T^{*}\omega_{R})$ and $\Theta(\omega_{R})^{\sigma}=\Theta(\omega_{R}^{\sigma})$. It follows that ${\mathcal M}_{R}^{\chi_{k}}={\mathcal M}_{1/R}^{\chi_{-k}}$ and that a subfield of ${\mathbb L}$ is a field of definition of $\omega_{R}$ if and only if it is a field of definition of $\Theta(\omega_{R})$. In this way, we may restrict to $k \geq 0$.
\end{rema}

\begin{rema}[Rational dynamical systems]
To each $R \in {\mathbb C}(z)$ we may associate 
a {\it rational dynamical system} $D_{R}: \; \dot{z}=R(z)$. If $T \in {\rm PSL}_{2}({\mathbb C})$ and $z$ is a holomorphic solution of the rational dynamical system $D_{R}$, then $u=T^{-1} \circ z$ is a holomorphic solution of the pull-back  $T^{*}D_{R}: \; \dot{u}=R(T(u))(T'(u))^{-1}$. In this way, to work with $D_{R}$ is equivalent to work with the meromorphic $(-1)$-form $R(z) dz^{-1}$. So, ${\mathcal M}_{R}^{\chi_{-1}}={\mathcal M}_{1/R}^{\chi_{1}}$ is the field of moduli of the rational dynamical system $D_{R}$. This relationship may provide algebraic tools (number theoretical tools) to study such rational dynamical systems.
\end{rema}

%%%%%%%%%%%%%%%
\subsection*{Question:}
Let $R \in {\mathbb L}(z)$. For each $k \in {\mathbb Z} \cup \{\infty\}$, we have defined the $\chi_{k}$-field of moduli ${\mathcal M}_{R_{\lambda}}^{\chi_{k}}$.
Is there a natural relationship between these different fields of moduli? The following example considers a particular case.

\begin{example} Let us return to Example \ref{ejemplo0}, where ${\mathbb L}={\mathbb C}$. In there, we consider, 
for each $\lambda \in {\mathbb C} \setminus \{0\}$, the quadratic map $R_{\lambda}(z)=z^{2}+\lambda$. If $k \in {\mathbb Z}$, we now consider the rational $k$-form
$\omega_{\lambda}=R_{\lambda}(z) dz^{k}=(z^{2}+\lambda) dz^{k}$.
Now, $R_{\lambda} \equiv_{\chi_{k}} R_{\mu}$ if and only if there is some $T \in {\rm PSL}_{2}({\mathbb C})$ such that $R_{\mu}(z)=R_{\lambda}(T(z))(T'(z))^{k}$.
As $\omega_{\lambda}$ has exactly one pole at $\infty$, it can be checked that $T(z)=az$, for some $a \in {\mathbb C} \setminus \{0\}$. In particular, 
$R_{\lambda} \equiv_{\chi_{k}} R_{\mu}$ if and only if there is some $a \in {\mathbb C}$ such that $a^{2+k}=1$ and $\mu=a^{k} \lambda$. This permits us to see that 
$U_{R_{\lambda}}^{\chi_{\infty}}=\{\sigma \in {\rm G}({\mathbb C}): \sigma(\lambda)=a^{k}\lambda, \; a^{2+k}=1\}$. This asserts that 
(i) ${\mathcal M}_{R_{\lambda}}^{\chi_{0}}={\mathbb Q}(\lambda)={\mathcal M}_{R_{\lambda}}^{\chi_{\infty}}$, (ii) ${\mathcal M}_{R_{\lambda}}^{\chi_{k}}={\mathbb Q}(\lambda^{2+k})$, if $k$ is odd, and
(iii) ${\mathcal M}_{R_{\lambda}}^{\chi_{k}}={\mathbb Q}(\lambda^{1+k/2})$, if $k$ is even.
\end{example}

%%%%%%%%%%%%%%%%%%%
\subsection{The ${\rm FOD/FOM}_{\chi_{k}}$-parameter}\label{Se:minimal}
In Theorem \ref{coromain}, we have seen that, under the assumption that ${\rm Aut}_{\chi_{k}}(R)$ is finite, and ${\mathcal M}_{R}^{\chi_{k}}$ is perfect (which holds under the assumption that ${\mathbb L}$ is either of characteristic zero or the algebraic closure of a finite field), then $R$ is $\chi_{k}$-definable either over ${\mathcal M}_{R}^{\chi_{k}}$ or a suitable quadratic extension of it. If $R$ has a degree of at least two, then ${\rm Aut}_{\chi_{k}}(R)$ is finite. If $R$ has a degree at most one, then its group of automorphisms might now not be finite.

\begin{lemm}\label{ejemplo1}
Let $R \in {\mathbb L}(z)$ of degree at most $1$ and $k \geq 0$ be an integer.
\begin{enumerate}[leftmargin=15pt]
\item If $k \neq 1$, then ${\mathcal M}_{R}^{\chi_{k}}$ is a field of definition, i.e., ${\rm FOD/FOM}_{\chi_{k}}(R)=1$.
\item If $k=1$, then ${\rm FOD/FOM}_{\chi_{k}}(R) \leq 2$.
\end{enumerate}
\end{lemm}
\begin{proof}
Let ${\mathcal Q}$ be the basic subfield of ${\mathbb L}$.

Let us first consider the case $k=0$.
If $R(z)=a_{0}$, then $R(T(z))=a_{0}$, for every $T \in {\rm PSL}_{2}({\mathbb L})$. So, for $\sigma \in {\rm G}({\mathbb L}/{\mathcal Q})$ it holds that $R^{\sigma} \equiv_{\chi_{0}} R$ if and only if $\sigma(a_{0})=a_{0}$. In other words, ${\mathcal M}_{R}^{\chi_{0}}={\mathcal Q}(a_{0})$ is a field of definition. If $R(z)$ has degree $1$, then we may use $T=R^{-1} \in {\rm PSL}_{2}({\mathbb L})$ to obtain that $R(T(z))=z$, i.e., ${\mathcal M}_{R}^{\chi_{0}}={\mathcal Q}$, and that it is a field of definition of $R$.

\medskip

Let us now consider the case $k \neq 0$. 

\subsubsection*{Case of degree zero}
If $R(z)=a_{0} \neq 0$ and $T(z)=z/\sqrt[k]{a_{0}}$, then $S(z)=R(T(z)) \left(T'(z)\right)^{k}=1$, in particular, $R \equiv_{\chi_{k}} 1$, i.e., $\omega_{R}$ is definable over ${\mathcal M}_{R}^{\chi_{k}}={\mathcal Q}$.

\subsubsection*{Case of degree one}
Let $R(z)=(a_{0}+a_{1}z)/(b_{0}+b_{1}z)$, where 
{\small
$$\left( \begin{array}{cc}
a_{0} & a_{1}\\
b_{0} & b_{1}
\end{array}
\right) \in {\rm SL}_{2}({\mathbb L}).$$
}

If $T_{1}=R^{-1}$, then $S(z)=R(T_{1}(z)) \left(T_{1}'(z)\right)^{k}=z/(a_{0}z-b_{0})^{2k}$.  
For $a_{0}=0$ and $T_{2}(z)=\sqrt[k+1]{b_{0}^{2k}}\; z$ it holds that $S(T_{2}(z)) \left(T_{2}'(z)\right)^{k}=z$, that is, $\omega_{R}$ is definable over ${\mathcal M}_{R}^{\chi_{k}}={\mathcal Q}$. 
Let us now assume $a_{0}\neq 0$.

\smallskip
\noindent
(1) If $b_{0} \neq 0$, then, by taking $T_{3}(z)=b_{0}z/a_{0}$, one has that $Q(z)=S(T_{3}(z)) \left(T_{3}'(z)\right)^{k}=z/a_{0}^{k+1}b_{0}^{k-1}(z-1)^{2k}$. In order to compute ${\mathcal M}_{R}^{\chi_{k}}={\mathcal M}_{Q}^{\chi_{k}}$, one needs to find those $\sigma \in {\rm G}({\mathbb L}/{\mathcal Q})$ for which 
there exists some $T_{\sigma} \in {\rm PSL}_{2}({\mathbb L})$ satisfying that
{\small
$$Q^{\sigma}(z)=Q(T_{\sigma}(z)) \left(T'_{\sigma}(z)\right)^{k}.$$
}

Since the only zero of $\omega_{Q}=Q(z) dz^{k}$ (an also of $\omega_{Q}^{\sigma}$) is $0$, the only simple pole of $\omega_{Q}$ (and also of $\omega_{Q}^{\sigma}$) is $\infty$ and the only pole of order $2k$ of $\omega_{Q}$ (and also of $\omega_{Q}^{\sigma}$) is $1$, we must have that $T_{\sigma}(0)=0$, $T_{\sigma}(\infty)=\infty$ and $T_{\sigma}(1)=1$, i.e., $T_{\sigma}(z)=z$, in particular, 
$Q^{\sigma}(z)=Q(z)$. This ensures that
$\sigma(a_{0}^{k+1}b_{0}^{k-1})=a_{0}^{k+1}b_{0}^{k-1}$ and, in particular, ${\mathcal M}_{R}^{\chi_{k}}={\mathcal Q}(a_{0}^{k+1}b_{0}^{k-1})$ is a field of definition of $\omega_{R}$.

\smallskip
\noindent
(2) If $b_{0}=0$, then $S(z)=1/a_{0}^{2k}z^{2k-1}$. To compute the field of moduli, in this case, we need to find those $\sigma \in {\rm G}({\mathbb L}/{\mathcal Q})$ for which there exists $T_{\sigma} \in {\rm PSL}_{2}({\mathbb L})$ so that $S^{\sigma}(z)=S(T_{\sigma}(z)) \left(T'_{\sigma}(z)\right)^{k}$. If $k \geq 2$, then necessarily $T_{\sigma}(0)=0$ and $T_{\sigma}(\infty)=\infty$ (since $0$ is a pole of order $2k-1$ and $\infty$ is a simple pole); that is, $T_{\sigma}(z)=\lambda_{\sigma}z$. But in this case the above condition on $T_{\sigma}$ ensures we should have $\sigma(a_{0}^{2k})z=a_{0}^{2k}z$, that is, $\sigma(a_{0}^{2k})=a_{0}^{2k}$. In this way, ${\mathcal M}_{R}^{\chi_{k}}={\mathcal Q}(a_{0}^{2k})$ is field of definition of $\omega_{R}$. Let us now assume $k=1$. In this case, we may also have the possibility $T_{\sigma}(0)=\infty$ and $T_{\sigma}(\infty)=0$ (since $0$ and $\infty$ are simple poles); that is, $T_{\sigma}(z)=\lambda_{\sigma}/z$. The asked condition on $T_{\sigma}$ ensures that $\sigma(a_{0}^{2})=-a_{0}^{2}$. It follows that, if there is no $\sigma$ satisfying that $\sigma(a_{0}^{2})=-a_{0}^{2}$ (for instance if $a_{0}^{2}=\sqrt[3]{2}$), then ${\mathcal M}_{R}^{\chi_{1}}={\mathcal Q}(a_{0}^{2})$ is field of definition of $\omega_{R}$. On the other hand, if there are such $\sigma$'s (for instance when $a_{0} \notin \overline{\mathcal Q}$ or when $a_{0}^{2} \in \overline{\mathcal Q}$ has minimal polynomial $P(z) \in {\mathcal Q}[z]$ with a factor $(z^{2}-a_{0}^{4})$ in ${\mathbb L}(z)$ such as the case for $a_{0}^{2}=\sqrt{2}$), then ${\mathcal M}_{R}^{\chi_{1}}={\mathcal Q}(a_{0}^{4})$ (a subfield of index $2$ of the field of definition ${\mathcal Q}(a_{0}^{2})$ of $\omega_{R}$).
\end{proof}

Summarizing al, the above, is the following.

\begin{coro}\label{grado2}
Let $k \in {\mathbb Z}$ and ${\mathbb L}$ be an algebraically closed field and $R \in {\mathbb L}(z)$.                  
If ${\mathcal M}_{R}^{\chi_{k}}$ is perfect, then ${\rm FOD/FOM}_{\chi_{k}}(R) \leq 2$.  In particular, this holds 
if ${\mathbb L}$ is either of characteristic zero or the algebraic closure of a finite field.
\end{coro}

%%%%%%%%%%%%
\subsection{Sufficient conditions for a rational map to be definable over its $\chi_{k}$-field of moduli}
In \cite{Silv1}, Silverman noted that every rational map, either of even degree or $\chi_{\infty}$-equivalent to a polynomial, can be defined over its $\chi_{\infty}$-field of moduli.
In the proof of Proposition \ref{ejemplo2}, we will provide examples of rational maps $R$ of even degree which cannot be definable over their corresponding $\chi_{k}$-fields of moduli. In these examples, the pole divisor of the form $\omega_{R}=R(z) dz^{k}$ has an even degree. The following result asserts that this is the only obstruction.

\begin{coro}\label{suficiente}
Let $R \in {\mathbb L}(z)$ of degree at least two, $k \in {\mathbb Z}$ and $\omega_{R}=R(z) dz^{k}$. 
If  ${\mathcal M}_{R}^{\chi_{k}}$ is perfect and the divisor of poles of $\omega_{R}$ has odd degree, then 
$R$ is $\chi_{k}$-definible over ${\mathcal M}_{R}^{\chi_{k}}$, that is, ${\rm FOD/FOM}_{\chi_{k}}(R)=1$. In particular, if 
$R$ is a polynomial, then it is $\chi_{k}$-definible over ${\mathcal M}_{R}^{\chi_{k}}$.
\end{coro}
\begin{proof}
As a consequence of Theorem \ref{teocociclo}, there is a $\chi_{k}$-cocycle $\{T_{\sigma}\}_{\sigma \in {\rm G}({\mathbb L}/{\mathcal M}_{\chi_{k})}}$. 
Following similarly as in the first part of the proof of Theorem \ref{teomain}, 
there is a smooth algebraic curve $B$ of genus zero, defined over ${\mathcal M}_{R}^{\chi_{k}}$, there is an isomorphism $L:{\mathbb P}^{1} \to B$, defined over ${\mathbb L}$, and there is some $S \in {\mathbb L}(B)$, also defined over ${\mathcal M}_{R}^{\chi_{k}}$ such that $R=\chi_{k}(L)(S)$.

If $D$ is the divisor of poles of $\omega_{R}$, then the image divisor $L(D)$ is a ${\mathcal M}_{R}^{\chi_{k}}$-rational divisor. Let us assume this divisor has an odd degree. As $B$ has genus zero and it is defined over ${\mathcal M}_{R}^{\chi_{k}}$, we may proceed similarly as in the proof of Lemma \ref{lemaaux} to obtain a ${\mathcal M}_{R}^{\chi_{k}}$-rational point.
Now, proceeding similarly as in the last part of the proof of Theorem \ref{teomain}, we may conclude that $R$ is $\chi_{k}$-definable over ${\mathcal M}_{R}^{\chi_{k}}$.

If $R(z)$ is a polynomial, then $\omega_{R}$ has a unique pole at $\infty$. In this case, we proceed as above using the divisor $D=[\infty]$.
\end{proof}

%%%%%%%%%%%%%%%%%%%%
%%%%%%%%%%%%%%%%%%%%
\subsection{Real $\chi_{k}$-field of moduli}\label{Sec:real}
In this section, we assume ${\mathbb L}={\mathbb C}$, and we discuss the situation when the field of moduli is a subfield of ${\mathbb R}$.  

Let us consider the complex conjugation $J(z)=\overline{z}$ (an anti-holomorphic automorphism of the Riemann sphere $\widehat{\mathbb C}$).
Let us set by $\tau_{J}(z)=\overline{z} \in {\rm G}({\mathbb C})$ (the complex conjugation seen as an automorphism of the field ${\mathbb C}$). In this way, if $R \in {\mathbb C}(z)$, then $R^{\tau_{J}}=J \circ R \circ J$.
We set the pull-back of $\omega_{R}=R(z) dz^{k}$ by the reflection $J$ as the meromorphic $k$-form $J^{*}\omega_{R}:=R^{\tau_{J}}(z) dz^{k}$ (note that this is not the usual pull-back by $J$). The pull-back of $\omega_{R}$ under the extended M\"obius transformation $U=T \circ J$, where $T \in {\rm PSL}_{2}({\mathbb C})$ is defined as 
{\small $$(T \circ J)^{*}\omega_{R}=J^{*}(T^{*}\omega_{R})=R^{\tau_{J}}(T^{\tau_{J}}(z)) ((T^{\tau_{J}})'(z))^{k} dz^{k}.$$
}

If $(J \circ T)^{*}\omega_{R}=\omega_{R}$, i.e., $R^{\tau_{J}}(T^{\tau_{J}}(z)) \left((T^{\tau_{J}})'(z)\right)^{k}=R(z)$, then we say that 
$T \circ J$ is an {\it anti-holomorphic automorphism} of $\omega_{R}$. 

For instance, by taking $T=I$, we see that $J$ is an anti-holomorphic automorphism of $\omega_{R}$ if and only if $R^{\tau_{J}}=R$, i.e., if and only if $R \in {\mathbb R}(z)$.

In the case of $\chi_{\infty}$, it is known that a rational map $R$ has its $\chi_{\infty}$-field of moduli inside ${\mathbb R}$ if and only if there an extended M\"obius transformation (i.e., the composition of a M\"obius transformation with the complex conjugation) that commutes with $R$, and it is definable of ${\mathbb R}$ if and only if such an extended M\"obius transformation can be chosen to be of order two and with fixed points (i.e., a reflection) \cite{Hidalgo}. Below, we obtain a similar result for  $\chi_{k}$.

\begin{prop}\label{notita}
Let $k \in {\mathbb Z}$, $R \in {\mathbb C}(z)$ and $\omega_{R}=R(z) dz^{k}$. Then 
\begin{enumerate}[leftmargin=15pt]
\item ${\mathcal M}_{R}^{\chi_{k}} \leq {\mathbb R}$ if and only if $\omega_{R}$ admits 
an anti-holomorphic automorphism.

\item $R$ is $\chi_{k}$-definable over ${\mathbb R}$ if and only if $\omega_{R}$ admits a reflection as an anti-holomorphic automorphism.

\end{enumerate}
\end{prop}

\begin{lemm}\label{modulireal}
Let $k \in {\mathbb Z}$ and $R \in {\mathbb C}(z)$.
\begin{enumerate}[leftmargin=15pt]
\item 
${\mathcal M}_{R}^{\chi_{k}} \leq {\mathbb R}$ if and only if there exists $T \in {\rm PSL}_{2}({\mathbb C})$ so that 
$R^{\tau_{J}}(z)=R(T(z)) \left(T'(z)\right)^{k}$.

\item $R$ is $\chi_{k}$-definable over 
${\mathbb R}$ if and only if there exists $T \in {\rm PSL}_{2}({\mathbb C})$ such that
{\small
$$R^{\tau_{J}}(z)=R(T_{\tau_{J}}(z)) \left(T_{\tau_{J}}' (z)\right)^{k}, \quad \mbox{$T_{\tau_{J}}:=T^{-1} \circ T^{\tau_{J}}$.}$$
}

\end{enumerate}
\end{lemm}
\begin{proof}
This is an special case of Corollary \ref{TeoWeil1} with ${\mathbb L}={\mathbb C}$ and  ${\mathbb K}={\mathbb R}$.
\end{proof}

\begin{rema}
Let us notice the difference between the conditions  (1) and (2) of Lemma \ref{modulireal}. The M\"obius transformation in (1) is of general form, but the one in (2) has a special form.
For instance, for $k=1$, $T_{\tau_{J}}(z)=(\lambda z +it)/(is z +\overline{\lambda})$, where $\lambda \in {\mathbb C}$, $t,s \in {\mathbb R}$ and $|\lambda|^{2}+ts=1$.
\end{rema}

\begin{proof}[Proof Proposition \ref{notita}]

(1) Lemma \ref{modulireal} asserts that ${\mathcal M}_{R}^{\chi_{k}}$ is a subfield of ${\mathbb R}$ if and only if there is some $T \in {\rm PSL}_{2}({\mathbb C})$ so that $R^{\tau_{J}}(z)=R(T(z))\left(T'(z)\right)^{k}$, equivalently, $R(z)=R^{\tau_{J}}\left(T^{\tau_{J}}(z)\right) \left(\left(T^{\tau_{J}}\right)'\right)^{k}$. It follows that
$(T \circ J)^{*} \omega_{R}=\omega_{R}$.
(2) By the definition, $R$ is $\chi_{k}$-definable over ${\mathbb R}$ if and only if there is some $T \in {\rm PSL}_{2}({\mathbb C})$ and some $S(z) \in {\mathbb R}(z)$ with $\omega_{R}=T^{*}\omega_{S}$, where $\omega_{S}=S(z) dz^{k}$. Since the reflection $J(z)=\overline{z}$ 
is an anti-holomorphic automorphism of $\omega_{S}$, then we may see that the reflection $T^{-1} \circ J \circ T$ is an anti-holomorphic automorphism of $\omega_{R}$. 
\end{proof}

%%%%%%%%%%%%%%%%%%
\subsection{Example 1} Let $k=1$ and $R(z)=z/(z-1)(z-e^{i\theta})$, so $\omega_{R}=R(z) dz$. As the reflection $U(z)=e^{i\theta}\overline{z}$ satisfies that $U^{*}\omega_{R}=\omega_{R}$, Proposition \ref{notita} asserts that $R$ is $\chi_{1}$-definable over ${\mathbb R}$. In fact, if we consider $T(z)=e^{i\theta/2}z$, then 
{\small
$$T^{*}\omega_{R}=\frac{e^{i\theta} z dz}{(e^{i\theta/2} z-1)(e^{i\theta/2}z-e^{i\theta})}=\frac{z dz}{z^{2}-2\cos(\theta/2)z+1}.$$
}

%%%%%%%%%%%%%%%%%%%
\subsection{Existence of rational maps of degree at least two which are not definable over their $\chi_{k}$-field of definition}

\begin{prop}\label{ejemplo2}
For each $k \in {\mathbb Z}$, there are rational maps $R\in {\mathbb C}(z)$, of degree $d \geq 2$, with ${\mathcal M}_{R}^{\chi_{k}} \leq {\mathbb R}$ but not $\chi_{k}$-definable over ${\mathbb R}$.
\end{prop}
\begin{proof}
As previously noticed in Remark \ref{negaposi}, we only need to consider $k \geq 0$.

(1) Let us first consider the case $k=1$.
Let $r>1$, $\theta \in {\mathbb R}$ be so that $e^{-2i\theta} \neq e^{2i\theta} \neq -1$ and $\lambda \in {\mathbb C}$ be so that $\lambda=-\overline{\lambda}e^{2i\theta}$. It can be seen (by using the cross-ratio) that the only subsets of cardinality $4$ of $\{-r^{2},0,1,\infty,-re^{i\theta},re^{i\theta}\}$ which are contained in a circle are given by
{\small
$$\{-r^{2},0,1,\infty \}, \quad
\{-re^{i\theta},0,re^{i\theta},\infty\}, \quad 
\{1,re^{i\theta},-r^{2},-re^{i\theta}\}.$$
}

There is a value $r_{\theta}$ so that if $r \neq r_{\theta}$, then the (classes of) cross-ratios of these three $4$-subsets are different (see details in \cite{Hidalgo:erratum}). Let us assume $r \neq r_{\theta}$.
Set
{\small
$$R(z)=\frac{(z-1)(z+r^{2})(z^{2}-r^{2}e^{2i\theta})}{\lambda z^{3}}$$
}
and $\omega_{R}=R(z) dz$ the corresponding meromorphic $1$-form. 

The zeroes of $\omega_{R}$ are all simple and given by: $1$, $-r^{2}$ and $\pm re^{i\theta}$; and the poles are $0$ and $\infty$, both of order $3$.
If we take $T(z)=-r^{2}/z$, then, since $\lambda=-\overline{\lambda}e^{2i\theta}$, we obtain $T^{*} \omega_{R}=\omega_{R}^{\tau}$. In this way, ${\mathcal M}^{\chi}_{R} \leq {\mathbb R}$.

To prove that $R$ is not $\chi_{k}$-definable over ${\mathbb R}$, we need to check (by Proposition \ref{notita}) that it has no reflection as an anti-holomorphic automorphism. Assume, by the contrary, that there is a reflection $U$ as an anti-holomorphic automorphism of $\omega_{R}$, i.e., 
{\small
$$U^{*} \omega_{R}=\frac{(z-1)(z+r^{2})(z^{2}-r^{2}e^{2i\theta}) dz}{\lambda z^{3}}.$$
}

Since the poles of $U^{*} \omega_{R}$ are $0$ and $\infty$, we must have that
$U(\{0,\infty\})=\{0,\infty\}.$
Since the zeroes of $U^{*} \omega_{R}$ are $1$, $-r^{2}$ and $\pm re^{-i\theta}$, we must have that $U(\{1, -r^{2}, -re^{-i\theta}, re^{-i\theta}\})=\{1, -r^{2}, -re^{i\theta}, re^{i\theta}\}$.

There is no reflection interchanging $0$ with $\infty$ and satisfying the above property (since for any reflection $U$, interchanging $0$ with $\infty$, must satisfy that $U(1)>0$). In this way, $U(z)=\mu \overline{z}$, where $|\mu|=1$. But again, to satisfy the above, one must have $\mu=1$, i.e., $U(z)=J(z)=\overline{z}$. But, as   $e^{-2i\theta} \neq e^{2i\theta}$,
{\small
$$J^{*}\omega_{R}=\frac{(z-1)(z+r^{2})(z^{2}-r^{2}e^{-2i\theta}) dz}{\lambda z^{3}} \neq \frac{(z-1)(z+r^{2})(z^{2}-r^{2}e^{2i\theta}) dz}{\lambda z^{3}},$$
}
a contradiction (so, $R$ is not $\chi_{k}$-definable over ${\mathbb R}$).

(2) Let us now consider the case $k>1$. Let  
{\small
$$R(z)=\left(\frac{(z-1)(z+r^{2})(z^{2}-r^{2}e^{2i\theta})}{\lambda z^{3}}\right)^{k},$$}
 with the same restrictions on the variables $r$, $\theta$ and $\lambda$ as before. The same arguments as above apply in this example.

(3) For $k=0$, we may use 
{\small $$R(z)=\frac{(z-1)(z+r^{2})(z^{2}-r^{2}e^{2i\theta})}{\lambda z^{2}}$$} 
but with $\lambda=\overline{\lambda} e^{2i\theta}$. Other examples (of Belyi rational maps) were provided in \cite{Couv}.
\end{proof}

%%%%%%%%%%%%%%%%%%
%%%%%%%%%%%%%%%%%%
\section{A final remark}\label{Sec:flat}
In some contexts, one might be interested in the pull-back up to scalars (for instance, if interested in flat structures on Riemann surfaces). In this setting, we may define the $P\chi_{k}$-action of ${\rm PSL}_{2}({\mathbb C})$ on ${\mathbb C}(z)$ by
$$S \equiv_{P\chi_{k}} R \iff  \exists\; T \in {\rm PSL}_{2}({\mathbb C}) \; \mbox{and} \; \exists\; \lambda \in {\mathbb C}^{*}={\mathbb C} \setminus \{0\} \;\mbox{such that} \;  \omega_{S}=\lambda T^{*} \omega_{R},$$
where, as before, $\omega_{R}=R(z) dz^{k}$ and $\omega_{S}=S(z) dz^{k}$. The field of moduli, in this case, is ${\mathcal M}_{R}^{P\chi_{k}}={\rm Fix}(\Gamma_{R}),$ where
$\Gamma_{R}=\{ \sigma \in {\rm G}({\mathbb C}): R^{\sigma} \equiv_{P\chi_{k}} R\}<{\rm G}({\mathbb C}).$

It is clear from the definitions that ${\mathcal M}_{R}^{P\chi_{k}} \leq {\mathcal M}_{R}^{\chi_{k}}$. In this case, instead of considering the group of automorphisms ${\rm Aut}(\omega_{R})$, we need to consider the group
$${\rm Aut}(\omega_{R})^{P}=\{ T \in {\rm PSL}_{2}({\mathbb C}): \exists\; \lambda_{T} \in {\mathbb C}^{*}, \; \omega_{R}=\lambda_{T} T^{*} \omega_{R}\}$$
which may not be a finite group; for instance, if $\omega=z^{2} dz$, then ${\rm Aut}(\omega)=\langle A(z)=e^{2 \pi i/3}z\rangle \cong {\mathbb Z}_{3}$ and ${\rm Aut}(\omega)^{P} \cong {\mathbb C}^{*}$. But, if the number of zeroes and poles is at least $3$, 
then this group is finite. The reason for this fact is the following. If $T^{*}\omega=\lambda \omega$, for some $\lambda \in {\mathbb C}^{*}$, then $T$ must preserve the set of zeroes and the set of poles of $\omega$. But a M\"obius transformation is uniquely determined by its action at $3$ different points. Under this assumption, we may follow the same arguments as in the previous section to obtain the following.

\begin{theo}\label{projective}
Let $R \in {\mathbb C}(z)$ and $k \in {\mathbb Z}$. If the number of zeroes and poles of the form $\omega_{R}=R(z) dz^{k}$ is at least $3$, then $R$ has a $\chi_{k}$-field of definition of degree at most $2$ over 
${\mathcal M}_{R}^{P\chi_{k}}$.
\end{theo}

%%%%%%%%%%%%%%%%%
\subsection{Example}
Let us consider a meromorphic $2$-form $\omega=R(z) dz^{2}$, where $R \neq 0$, over $\widehat{\mathbb C}$. Let $E_{\omega}$ be the (finite) set of zeroes and poles of $\omega$ and let us consider the planar Riemann surface $S_{\omega}=\widehat{\mathbb C} \setminus E_{\omega}$, which has signature $(0;\infty, \stackrel{\#E_{\omega}}{\cdots}, \infty)$.
We assume the poles of $\omega$ to be simple ones (that is, $\omega$ is integrable on $S_{\omega}$).  A flat structure on $S_{\omega}$ is induced by $\omega$: local coordinates for $z_{0} \in S_{\omega}$ have the form $\rho(z)=\int_{z_{0}}^{z} \sqrt{R(z)} dz$. The local change of coordinates has the form $z \mapsto \pm z +c$, where $c \in {\mathbb C}$. 
If $\lambda \in {\mathbb C}^{*}$, then $\lambda \omega$ provides a new system of coordinates; they are given by amplification of the originals by $\sqrt{\lambda}$ (i.e. the change of coordinates still of the above form) and in particular, it provides the same flat structure on $S_{\omega}$ given by $\omega$. 
The flat structures defined by the meromorphic $2$-forms $\omega_{1}$ and $\omega_{2}$ are equivalent if there exists a M\"obius transformation $T \in {\rm PSL}_{2}({\mathbb C})$ so that $T:S_{\omega_{1}} \to S_{\omega_{2}}$ is (in the corresponding induced flat structures) affine. This is equivalent to have $T^{*}\omega_{2}=\lambda \omega_{1}$ for a suitable $\lambda \in {\mathbb C}^{*}$.

If $\sigma \in {\rm G}({\mathbb C})$, then we may consider the meromorphic $2$-form $\omega^{\sigma}=R^{\sigma}(z) dz^{2}$. We note that $E_{\omega^{\sigma}}$ and $E_{\omega}$ have the same cardinality. So, $S_{\omega^{\sigma}}$ is a new Riemann surface with a flat structure of the same signature as $S_{\omega}$. It could be that these two flat structures are not equivalent, that is, there might not be a M\"obius transformation $T \in {\rm PSL}_{2}({\mathbb C})$ with $\omega^{\sigma}=\lambda_{\sigma} T^{*}\omega$. The field of moduli of the flat structure $S_{\omega}$ is the field  ${\mathcal M}_{R}^{P\chi_{2}}$. Theorem \ref{projective} ensures that this flat structure is definable over an extension of degree at most $2$ over ${\mathcal M}_{R}^{P\chi_{2}}$ if the cardinality of $E_{\omega}$ is at least $3$.

%%%%%%%%%%%%%%%%%
\subsubsection{}
If $\omega=R(z) dz^{2}$ has $E_{\omega}$ of cardinality equal to $3$, then, up to a M\"obius transformation, we may assume
$E_{\omega}=\{\infty,0,1\}$. We may also assume that $\infty$ is a pole (since the degree of a quadratic form is $-4$). In this way, up to scalars, $R(z)=z^{a}(z-1)^{b}$, where $a,b \in {\mathbb Z}-\{0\}$ with $a+b>-4$. In particular, ${\mathcal M}_{R}^{P\chi_{2}}={\mathbb Q}$ is a $\chi_{2}$-field of definition of $R$.

\subsubsection{}
If $\omega=R(z) dz^{2}$ has $E_{\omega}$ of cardinality equal to $4$, then, up to a M\"obius transformation, we may assume
$E_{\omega}=\{\infty,0,1, \mu\}$, where $\mu \in {\mathbb C} \setminus \{0,1\}$ and that $\infty$ is a pole. In this way, up to scalars, $R(z)=z^{a}(z-1)^{b}(z-\mu)^{c}$, where $a,b,c \in {\mathbb Z} \setminus \{0\}$ with $a+b+c>-4$. In particular, ${\mathcal M}_{R}^{P\chi_{2}}<{\mathbb Q}(\mu)$. To compute ${\mathcal M}_{R}^{P\chi_{2}}$, we need to determine those $\sigma \in {\rm G}({\mathbb C})$ for which there exists a M\"obius transformation $T_{\sigma}$ and some $\lambda_{\sigma} \in {\mathbb C}^{*}$ with 
$T^{*}\omega=\lambda_{\sigma}\omega^{\sigma}=\lambda_{\sigma} z^{a}(z-1)^{b}(z-\sigma(\mu))^{c} dz^{2}$. This obligates us to have 
$$T_{\sigma}: \{\infty,0,1,\sigma(\mu)\} \mapsto \{\infty,0,1,\mu\}.$$

Also, there are restrictions in the sense that $T_{\sigma}$ should send a zero (respectively, a pole) to a zero (respectively, a pole), preserving the orders. The above obligates to have $T_{\sigma} \in {\mathbb G}=\langle A(z)=1/z, B(z)=\sigma(\mu)/(\sigma(\mu)-z)\rangle \cong {\mathfrak S}_{3}$. This ensures that 
${\mathbb Q}(j(\mu))<{\mathcal M}_{R}^{P\chi_{2}}<{\mathbb Q}(\mu),$
where $j(\mu)=(1-\mu+\mu^{2})^{3}/\mu^{2}(1-\mu)^{2}$ is the classical elliptic $j$-function. In particular, ${\mathbb Q}(\mu)$ is an extension of degree either $1$, $2$, $3$ or $6$ over ${\mathcal M}_{R}^{P\chi_{2}}$. The situation will depends on the values of $a,b,c$; for instance, if $a>0$, $b>0$ and $c>0$ are different, then $T_{\sigma}(z)=z$ and $\sigma(\mu)=\mu$, obtaining the equality ${\mathcal M}_{R}^{P\chi_{2}}={\mathbb Q}(\mu)$. Another extremal case is when $a=b=c=-1$ (i.e. the four points of $E_{\omega}$ are simple poles), then ${\mathcal M}_{R}^{P\chi_{2}}={\mathbb Q}(j(\mu))$.

%%%%%%%%%%%%%%%%%%%
%%%%%%%%%%%%%%%%%%%

\end{document}